\documentclass[12pt]{article}

\usepackage[lmargin = 2cm, rmargin = 2cm, bmargin = 2cm,tmargin=2cm]{geometry}

\usepackage{titlesec}

\usepackage{microtype}
  
\usepackage{enumitem}

\usepackage{xfrac}

\usepackage{subfig}

\usepackage{titlesec}

\usepackage{graphicx}

\usepackage{amssymb}

\usepackage{amsmath}

\usepackage{cancel}

\usepackage{amsthm}

\usepackage{accents}

\usepackage{wrapfig}

\usepackage{mathrsfs}

\usepackage{eufrak}

\usepackage{tikz-cd}

\usepackage{xcolor}

\usepackage{hyphenat}

\usepackage{fancyhdr}

\usepackage{url}

\usepackage{mathdots}

\usepackage{etoolbox}

\usepackage{tgtermes}
\definecolor{ggreen}{rgb}{0,0.75,0.08}

\theoremstyle{plain}

\theoremstyle{definition}

\newtheorem{theorem}{Theorem}[section]

\newtheorem{lemma}[theorem]{Lemma}

\newtheorem{corollary}[theorem]{Corollary}

\newtheorem{definition}[theorem]{Definition}

\newcounter{dummy3} 
\newtheorem{thm3}[dummy3]{Theorem}

\newtheorem{cnj3}[dummy3]{Conjecture}

\newtheorem{question*}{Question}

\newtheorem{notation}[theorem]{Notation}

\newcommand{\A}{\ensuremath{\alpha}}

\newcommand{\K}{\ensuremath{\kappa}}
\newcommand{\B}{\ensuremath{\beta}}
\newcommand{\si}{\ensuremath{\sigma}}

\newcommand{\W}{\ensuremath{\omega}}

\newcommand{\RR}{\ensuremath{\mathbb R}}

\newcommand{\II}{\ensuremath{\mathbb I}}
\newcommand{\KK}{\ensuremath{\mathbb K}}
\newcommand{\JJ}{\ensuremath{\mathbb J}}
\newcommand{\LL}{\ensuremath{\mathbb L}}

\newcommand{\NN}{\ensuremath{\mathbb N}}
\newcommand{\HH}{\ensuremath{\mathbb H}}
\newcommand{\0}{\ensuremath{\varnothing}}

\newcommand{\cL}{\ensuremath{\mathcal L}}

\newcommand{\cD}{\ensuremath{\mathcal D}}

\newcommand{\cS}{\ensuremath{\mathcal S}}

\newcommand{\cK}{\ensuremath{\mathcal K}}

\newcommand{\cC}{\ensuremath{\mathcal C}}

\newcommand{\cI}{\ensuremath{\mathcal I}}

\newcommand{\cU}{\ensuremath{\mathcal U}}

\newcommand{\DD}{\ensuremath{\partial}}

\newcommand{\cn}{\ensuremath{\frak c}}

\begin{document}

\openup 0.6em

\fontsize{13}{5}
\selectfont

	\begin{center}\LARGE The Shore Point Existence Problem is Equivalent to the Non-Block Point Existence Problem
	\end{center}
	
	\begin{align*}
	\text{\Large Daron Anderson }  \qquad \text{\Large Trinity College Dublin. Ireland }  
	\end{align*} 
	\begin{align*} \text{\Large andersd3@tcd.ie} \qquad \text{\Large Preprint February 2019}  
	\end{align*}$ $\\

	\begin{center}
		\textbf{ \large Abstract}
	\end{center}
	
	\noindent  We prove the three propositions are equivalent:
		$(a)$ Every Hausdorff continuum has two or more shore points.
		$(b)$ Every Hausdorff continuum has two or more non-block points.
		$(c)$ Every Hausdorff continuum is coastal at each point.
		Thus it is consistent that all three properties fail.
		We also give the following characterisation of shore points: The point $p$ of the continuum $X$ is a shore point if and only if there is a net of subcontinua in $\{K \in C(X): K \subset \K(p) - p\}$
		tending to $X$ in the Vietoris topology. This contrasts with the standard characterisation which only demands the net elements be contained in $X-p$.
		In addition we prove every point of an indecomposable continuum is a shore point.

\section{Introduction}

\noindent Leonel \cite{Leonel01} has improved the classic non-cut point theorem of Moore \cite{noncutmoore} 
by showing every metric continuum has two or more shore points.
Bobok, Pyrih and Vejnar \cite{B} observed Leonel's two shore points have the stronger property of being non-block points.

In \cite {Me2} the author proved it is consistent the result fails to generalise to Hausdorff continua.
Under Near Coherence of Filters (NCF) the Stone-\v Cech remainder $\HH^*$ of the half-line lacks non-block points and hence lacks coastal points.

This left open the question of whether there is a consistent example of a Hausdorff continuum without shore points.
This paper gives a positive answer.
Indeed we show the shore point and non-block point existence problems are equivalent.
They are also equivalent to a number of other problems involving shore, non-block, and coastal points of Hausdorff continua.

We also prove every shore point $p \in X$ has the stronger property of being a \textit{proper shore point}.
That means there is a net of subcontinua in the hyperspace $\{K \in C(X): K \subset \K(p) - p\}$ tending to $X$ in the Vietoris topology.
This is not apparent from the definition of a shore point, which only requires the net elements be contained in $X-p$.

\section{Terminology and Notation}

\noindent For sets $A$ and $B$ define $A - B = \{ a \in A \colon a \notin B\}$. 
For $B = \{b\}$ we write $A-b$ without confusion. 
For $A \subset B$ we do not presume $A$ is a proper subset of $B$.
For a subset $S \subset X$ denote by $S^\circ$ and $\overline S$ the interior and closure of $S$ respectively. 
The \textit{boundary} of $S$ means the set $\DD S = \overline S \cap \overline{(X-S)}$.

Throughout $X$ is a continuum. 
That is to say a nondegenerate compact connected Hausdorff space.
For background on metric continua see \cite{kur2} and \cite{nadlerbook}.
The results cited here have analagous proofs for non-metric continua.

Throughout all maps are assumed to be continuous.
The map $f:X \to Y$ of continua is called \textit{monotone} to mean $f^{-1}(y) \subset X$ is connected for each $y \in Y$.
Theorem 6.1.28 of \cite{Engelking} says moreover $f^{-1}(K) \subset X$ is a continuum for each subcontinuum $K \subset Y$.

For $a,b \in X$ we call $X$ \textit{irreducible} about $\{a,b\}$ to mean $\{a,b\}$ is not contained in a proper subcontinuum of $X$.
The subspace $A \subset X$ is called a \textit{semicontinuum} to mean
for each $a,c \in A$ some subcontinuum $K \subset A$ has $\{a,c\} \subset K$.
Every subspace $A \subset X$ is partitioned into maximal semicontinua called the \textit{continuum components} of $A$.

For $N\ge 2$ we say the subcontinua $X_1, \ldots , X_N \subset X$ form a \textit{decomposition} 
and write $X_1  \oplus \ldots \oplus X_N$ to mean $X_1 \cup \ldots \cup X_N =X$
and no $X_n$ is contained in the union of the others.
We call $X$ \textit{decomposable} to mean it admits a decomposition and \textit{indecomposable} otherwise.
The latter is equivalent to admitting no decomposition with $N=2$ 
and equivalent to each proper subcontinuum being nowhere dense.
We say $X$ is \textit{hereditarily indecomposable} to mean its every subcontinuum is indecomposable.
Equivalently each pair of subcontinua are either disjoint or nested.

The \textit{composant} $\K(x)$ of the point $x \in X$ is the union of all proper subcontinua that have $x$ as an element.
Indecomposable metric continua are partitioned into $\cn$ many pairwise disjoint composants \cite{Ccomposants}.
In case $\K(x) \ne \K(y)$ then $X$ is irreducible about $\{x,y\}$.
% There exist indecomposable non-metric continua with exactly one composant, henceforth called \textit{Bellamy continua}.

By \textit{boundary bumping} we mean the principle that, for each proper closed $E \subset X$, each component $C$ of $E$ meets the boundary $\DD E = \overline E \cap \overline {X-E}$.
For the non-metric proof see $\S$47, III Theorem 2 of \cite{kur2}. 
One corollary of boundary bumping is that any $p \in X$ is in the closure of each continuum component of $X-p$.

Throughout $C(X)$ is the set of subcontinua of $X$.
We call $p \in X$ a \textit{shore point} to mean for each finite collection of open sets $U_1,U_2, \ldots U_n \subset X$
some subcontinuum $K \subset X-p$ meets each $U_m$.
This is equivalent to some net in $\{K \in C(X): K \subset X - p\}$ tending to $X$ in the Vietoris topology.
We do not need the full definition of the Vietoris topology here.

We call $p \in X$ a \textit{cut point} to mean $X-p$ is disconnected and a \textit{non-cut point} otherwise.
Clearly each shore point is non-cut.
We call $p \in X$ a \textit{non-block point} to mean $X-p$ has a dense continuum component.
Every non-block point is a shore point but the converse fails in general \cite{B}.
We call $x \in X$ a \textit{coastal point} to mean $x$ is an element of some proper dense semicontinuum.
Clearly $X$ has a non-block point if and only if it has a coastal point.

Theorem 5 of \cite{Bing01} says every point of a metric continuum is coastal.
The result generalises to separable continua \cite{Me1} but not to Hausdorff continua.
Under the set-theoretic axiom Near Coherence of Filters the non-metric continuum $\HH^*$ lacks non-block points and hence lacks coastal points \cite{Me2}.

\section{Equivalent Problems}

\begin{definition}

The continuum $X$ is called \textit{partially coastal} to mean there are points $x,y \in X$ with $x$ coastal and $y$ non-coastal.

\end{definition}

\begin{thm3}\label{equiv} The following propositions are equivalent.

\begin{enumerate}
	
	\item There exists a continuum without coastal points.
	\item There exists a continuum with exactly one non-coastal point.
	\item There exists a partially coastal continuum.
	\item There exists a continuum without shore points.
	\item There exists a continuum with exactly one shore point
	\item There exists a continuum without non-block points.
	\item There exists a continuum with exactly one non-block point.

\end{enumerate}

\end{thm3}

\begin{proof}
$(6)\implies(1)$ because $X$ has a non-block point if and only if it has a coastal point.
$(2) \implies (3)$ follows from how every continuum has more than one point.
$(4) \implies (6)$ follows from how every non-block point is a shore point.

$(1) \implies (2):$ Suppose $X$ has no coastal points. 
Let $[0,1]$ be an arc and the continuum  $Y$ be obtained by gluing $1 \in [0,1]$ to any fixed $p\in X$.
Identify $X$ and $[0,1]$ with their images in $Y$.
The fact each $x \in Y -0$ is coastal is witnessed by
the dense semicontinuum $S = \big (\bigcup \{(1/n, 1] : n \in \NN \}\big ) \cup X = Y - 0$.

To see $0 \in Y$ is non-coastal suppose otherwise.
That means $ 0 \in S \subset Y-q$ for some dense proper semicontinuum $S \subset Y$ and point $q \in Y$.
It is easy to see $[0,1] \subset S$. Hence $q \notin [0,1]$ and $q \in X-p$.

Let $Q:Y \to X$ be the monotone map that collapses $[0,1]$ to the point $p \in X$ and leaves the points of $X$ fixed.
Then $Q(S) \subset X$ is a dense semicontinuum with $p \in Q(S) \subset X-q$ contradicting the assumption that $p \in X$
is non-coastal.

$(1) \implies (5):$ Suppose $X$ has no coastal points and define $Y$, $Q$ and $S$ as before.
The dense semicontinuum $S$ witnesses how $0 \in Y$ is a non-block point hence a shore point.

To see $0$ is the only shore point first observe each $q \in (0,1]$ is a cut point hence not a shore point.
Now suppose some $q \in Y - [0,1]$ is a shore point.
Fix the open set $U = (0,1)$.
By assumption, for every collection of open sets $U_1, \ldots , U_n \subset X-1$, some subcontinuum $K \subset  Y-q$ meets $U$ and all $U_m$.
Since $K$ meets $U$ and $U_1$ we have $1 \in K$ and hence $p \in Q(K)$.

Now allow $\{U_1,\ldots , U_n\}$ to range over all finite collections of open subsets of $X$.
The subcontinua $Q(K) \subset X - q$ constitute a dense proper semicontinuum of $X$.
Therefore $p \in X$ is coastal contrary to assumption.
We conclude $q$ is not a shore point as required.
The proof of $(1) \implies (7)$ is identical.

$(3) \implies (4):$ Suppose some $p \in X$ is non-coastal. 
Take two disjoint copies $X_1$ and $X_2$ of $X$.
Let $Y$ be the continuum obtained from $X_1 \sqcup X_2$ by identifying the points $a \in X_1$ and $b \in X_2$ corresponding to $p \in X$.
Identify $X_1$ and $X_2$ with their images in $Y$ and write $p \in Y$ for the shared image of $a$ and $b$.
We claim $Y$ has no shore points.

Observe $Y - p$ is homeomorphic to the disjoint union of $X_1 - a$ and $X_2 - b$. 
Hence $p$ is a cut point and not a shore point.
Now suppose $q \in Y-p$ is a shore point.
Without loss of generality $q \in X_2$.
Fix some open set $U \subset X_1 - p$.
For every collection of open sets $U_1, \ldots , U_n \subset X_2-p$ some subcontinuum $K \subset  Y-q$ meets $U$ and all $U_m$.
Since $K$ meets $U$ and $U_1$ we must have $p \in K$.
Moreover $K-p$ is disconnected and each component lies in one of $X_1$ or $X_2$.

Let $\cK$ be the family of components that lie in $X_2$.
Boundary bumping says $L \cup \{p\}$ is a continuum for each $L \in \cK$.
It follows $K \cap X_2 = \bigcup \big \{ \{p\} \cup L : L \in \cK \big \}$ is a subcontinuum of $X_2 - q$ that meets all $U_1,\ldots , U_n$.
Now allow $\{U_1,\ldots , U_n\}$ to range over all finite collections of open subsets of $X_2$.
Then the union $M$ of all $K \cap X_2$ is a dense semicontinuum of $X_2 - q$.
Let $Q: Y \to X_2$ be that map that compresses $X_1$ to the point $b \in X_2$.
Then the semicontinuum $Q(M)$ contradict how $X_2$ is not coastal at $b \in X_2$.
We conclude no $q \in Y$ is a shore point as required.

$(5) \implies (1):$ Suppose $p \in X$ is the unique shore point.
We claim $p$ is not coastal,
For then $p \in S \subset X-q$ for some $q \in X$ and dense proper semicontinuum $S$.
This implies $q$ is a non-block point and hence a shore point contrary to assumption.
This implies either $(1)$ or $(3)$ which we have shown are equivalent.

$(7) \implies (3):$ Suppose $p \in X$ is the unique non-block point.
Since there is a non-block point there is also a coastal point.
But $p \in X$ cannot be coastal as this would imply some $q \in X-p$ is a non-block point.
hence $X$ is partially coastal.

\end{proof}

Under the set-theoretic axiom Near Coherence of Filters (NCF) the Stone-\v Cech remainder $\HH^*$ 
of the half-line is a continuum with no coastal points \cite{Me2}. From this we get the corollary.

\begin{corollary}
Propositions $(1)-(7)$ from Theorem \ref{equiv} are consistent.
\end{corollary}

For each continuum $X$ with a non-coastal point $p$ the author has shown \cite{Me1} there is a subcontinuum $M \subset X$ with $p \in M^\circ$
such that $X/M$ indecomposable and non- coastal at $M$.  
This raises the question whether $(1)-(7)$ are equivalent if we demand the continuum in question be indecomposable.

Section 5 shows $(4)$ and $(5)$ never hold for indecomposable continua.
In particular Theorem 3 says every point of an indecomposable continua is a shore point.

For $X$ indecomposable the remaining propositions seem more difficult and the methods of Theorem 1 no longer apply.
For example if $X$ is indecomposable with $(1)$ we cannot simply attach an arc to prove $(2)$ 
as the quotient space is manifestly decomposable.

In the other direction $X$ having $(7)$ does not imply the same for $X/M$.
For example assume NCF and glue any point $p \in \HH^*$ to the endpoint $1 \in [0,1]$.
Let $X$ be the quotient space.
Clearly for $X/M$ to be indecomposable $M$ is the union of the arc and some proper subcontinuum $K \subset \HH^*$ with $p \in K$.
Since the map $\HH^* \to \HH^*/K$ is monotone it follows $X/M \cong \HH^*/K$ lacks non-block points and fails $(7)$.

Before proceeding to Section 5 we take a diversion to prove all shore points are proper shore points.
Some of the terminology and results from Section 4 will be needed for Section 5.

\section{Proper Shore and Non-Block Points}

\begin{definition}
Recall $p \in X$ is called a \textit{shore point} to mean for each finite collection of open sets $U_1,U_2, \ldots , U_N \subset X$
some subcontinuum $K \subset X-p$ meets each $U_n$.
We call $p$ a \textit{proper shore point} to mean $K$ can always be chosen as a subset of $\K(p)-p$.
Otherwise we call $p$ a \textit{trivial shore point}.
\end{definition}

\begin{definition}
Recall $p \in X$ is called a \textit{non-block point} to mean $X-p$ has a dense continuum component.
We call $p$ a \textit{proper non-block} point to mean $\K(p)-p$ has a dense continuum component.
Otherwise we call $p$ a \textit{trivial non-block point}.
\end{definition}

This section shows there exist both proper and trivial non-block points but there is no such thing as a trivial shore point. Thus being a proper non-block point is a meaningful notion but being a proper shore point is not.
Hereditarily indecomposable metric continua provide an example of the former.

\begin{lemma}
Suppose $M$ is a hereditarily indecomposable metric continuum. 
Each $p \in M$ is a trivial non-block point.
\end{lemma}

\begin{proof}
Recall $M$ has $\cn$ many pairwise disjoint composants each of which is a dense semicontinuum \cite{Ccomposants}.
Each composant other than $\K(p)$ witnesses how $p$ is a non-block point.
It remains to show each continuum component $C$ of $\K(p)-p$ is nowhere dense.

Let $q \in C$ be arbitrary.
We can write $C = \bigcup \cC$ as a union of proper subcontinua with $q \in D$ for each $D \in \cC$.
Since $q \in \K(p)$ there exists a proper subcontinuum $K$ with $\{p,q\} \subset K$.
Clearly $K$ meets but is not contained in each $D \in \cC$.
Hereditary indecomposability implies each $D \subset K$ and therefore $C = \bigcup \cC \subset K$.
Since $X$ is indecomposable $K$ and thus $C$ is nowhere dense as required.
\end{proof}

Observe we only required $M$ to have more than one composant.
The problem is open whether there exists a hereditarily indecomposable Hausdorff continuum
with exactly one composant. 
For obstructions to finding such a space see Smith 
\cite{SmithHIRemainderProd, SmithHILexProd, SmithHISouslinArcsIL, SmithHISouslinProd}. 

For $\K(p) = X$ clearly the shore point $p$ is proper.
Henceforth assume  $X$ is irreducible about some $\{p,q\}$.
We first treat the case when $X$ is decomposable.

\begin{lemma}\label{dec}
Suppose $X$ is decomposable and irreducible about $\{p,q\}$. 
Then $p$ and $q$ are proper non-block (shore) points.
\end{lemma}

\begin{proof}

By assumption we can write $X = A \cup B$ as the union of two proper subcontinua.
Since $X$ is irreducible we have without loss of generality $p \in A \subset X-q$ and and $q \in B \subset X-p$.

Choose any $x \in A \cap B$. It follows from boundary bumping that $p$ is in the closure of the continuum component $C$ of $x$ in $A-p$.
Likewise $q$ is in the closure of the continuum component $D$ of $x$ in $B-q$.
Then $x \in C \cup D \subset X-\{p,q\}$ and $C \cup D$ is continuumwise connected.

Observe the subcontinuum $\overline {C \cup D}$ contains $\{p,q\}$.
By irreducibility we have $\overline {C \cup D} = X$ hence $C \cup D$ is dense.
To see $C \cup D \subset \K(p)$ recall $p \in A \subset X-q$ hence $A \cup D \subset \K(p)$. 
Therefore the subset $C \cup D$ of $A \cup D$ witnesses how $p$ is a proper non-block point. By symmetry the same argument applies to $q$. 
\end{proof}

We now deal with indecomposable $X$.
The definition allows finer control over why a given $p \in X$ fails to be a shore point.

\begin{definition}
Suppose $p \in X$ and $\cU = \{U_1,\ldots , U_{n+1}\}$ are open subsets of $X$.
We say $p$ \textit{disrupts} $\cU$ to mean no continuum component of $\K(p)-p$ meets all $U_1,\ldots , U_{n+1}$.
We say $p$ \textit{trivially disrupts} $\cU$ to mean there are distinct elements 
$U^1,U^2, \ldots, U^n \in \cU$ and nonempty open sets $V_i \subset U^i$ such that $p$ disrupts $\{V_1,\ldots V_n\}$.
Otherwise we say $p$ \textit{properly disrupts} $\cU$.
\end{definition}

For $p$ to properly disrupt $\cU = \{U_1,U_2, \ldots, U_n\}$ the definition requires $\cU$ be pairwise disjoint.
It is also clear the shore point $p\in X$ being trivial is equivalent to disrupting some nondegenerate family.
By boundary bumping no $p \in X$ can disrupt a family with only one element.
So if $p$ disrupts $\cU = \{U_1,U_2\}$ then $p$ properly disrupts $\cU$.

By induction it follows if $p$ disrupts $\cU = \{U_1,U_2, \ldots, U_n\}$
there exist $m \le n$ and elements $U^1,U^2, \ldots, U^m \in \cU$ and open sets $V_i \subset U^i$ such that $p$ properly disrupts $\{V_1,\ldots V_m\}$.
The next lemmas take advantage of that fact.

\begin{lemma}\label{proper}
Suppose $p \in X$ properly disrupts $\{U_1,\ldots, U_{n}\}$ and let $r \le n$ be fixed.
There exists a family $\cK(r)$ of subcontinua of $ \K(p) - p$ that satisfies both properties below.
\begin{enumerate}
	\item Each $K \in \cK(r)$ meets $U_m$ for each $m \ne r$.
	\item $\bigcup \cK(r)$ is dense in $U_m$ for each $m \ne r$.
\end{enumerate}
\end{lemma}

\begin{proof}
Without loss of generality $r=1$.
For each $m  \ne 1$ let $V_m \subset U_m$ be an arbitrary open subset.
Since $p$ properly disrupts $\{U_1,\ldots, U_{n}\}$ it cannot disrupt the family $ \{V_2,\ldots ,V_{n}\}$.
That means there is a subcontinuum \mbox{$K \subset \K(p)-p$} that meets each of $V_2,\ldots , V_{n}$.
Now let $V_2, \ldots , V_n$ range over the open subsets of $U_2, \ldots , U_n$. 
The union of all continua $K$ is dense in each $U_m$.
\end{proof}

\begin{lemma}\label{indec}
Suppose $X$ is indecomposable. Each shore point $p \in X$ is a proper shore point.
\end{lemma}

\begin{proof}
Suppose to the contrary that $p \in X$ is a trivial shore point.
That means $p$ disrupts some family $\cU = \{U_1, \ldots U_n\}$ of open subsets.
By discarding some elements of $\cU$ if necessary we can assume $p$ properly disrupts $\cU$.
Let $\cK(1)$ be a family of subcontinua of $ \K(p)-p$ as described in Lemma \ref{proper}.

Each $K \in \cK(1)$ has a continuum component $C(K)$ in $\K(p)-p$.
Since $p$ disrupts $\cU$ we know $C(K)$ is disjoint from $U_1$.
Recall $\overline {C(K)}$ is a subcontinuum.
It follows from boundary bumping $p \in \overline {C(K)} \subset X - U_1$.

As $K$ ranges over the elements of $\cK(1)$ the set $S = \bigcup \big \{\overline {C(K)}: K \in \cK(1) \big \}$ constitutes a semicontinuum with $p \in S \subset X - U_1$.
Therefore \mbox{$\overline S \subset X-U_1$} is a proper subcontinuum.
But Property (2) of Lemma \ref{proper} says $U_2 \subset \overline S$.
Thus \mbox{$\overline S \subset X$} is a proper subcontinuum with nonvoid interior.
Since $X$ is indecomposable this cannot occur. We conclude $p$ is a proper shore point.
\end{proof}

Theorem 2 follows from Lemmas \ref{dec} and \ref{indec}.

\begin{thm3}\label{theorem2}
Shore points are the same as proper shore points.
\end{thm3}

\section{Continua Without Shore Points}

\noindent 
The proof of Theorem \ref{equiv} says we can build a continuum $X$ without shore points by assuming
NCF and gluing two copies of $\HH^*$ together at a single point.
This section is about whether arbitrary continua $X$ without shore points come about this way 
$-$ by joining together continua without non-block points.
We first show $X$ is decomposable.

%\begin{thm3}\label{theorem3}
% Suppose $X$ has no shore points and $p \in X$.
% There is a decomposition $X = X_1 \oplus \ldots \oplus X_n$ with each $X_m$ non-coastal at $p \in X_m$.
% There are subcontinua $I_m \subset X_m$ with $p \in I_m$ such that
% $X/(I_1 \cup \ldots \cup I_n) = (X_1/I_1) \oplus \ldots \oplus (X_n/I_n)$ and each $X_m/I_m$ is Bellamy and non-coastal at $I_m$.
%\end{thm3}

\begin{lemma}\label{ndecomp}
Suppose $p \in X$ is not a shore point. 
There is a decomposition $X = X_1  \oplus \ldots \oplus X_n$ with $p \in X_1 \cap \ldots \cap X_n$.
\end{lemma}

\begin{proof}
Since $p$ is not a shore point it disrupts some family $\cU = \{U_1, \ldots U_n\}$ of open subsets.
Like before, we can assume $p$ properly disrupts $\cU$.
For each $r \le n$ let $\cK(r)$ be a family of subcontinua of $\K(p)-p$ as described in Lemma \ref{proper}.

The proof of Lemma \ref{indec} shows each $K(r) = \overline{ \bigcup \cK(r) }$ is a subcontinuum 
with $p \in K(r) \subset X-U_r$ and $U_m \subset K(r)$ for each $m \ne r$. Let $\cL(r)$ be the family of all subcontinua with both properties. Without loss of generality we can assume $K(r) = \overline{ \bigcup \cL(r) }$ is the largest subcontinuum with both properties.

%Thus Y = K(1) \cup K(2) \cup \ldots \cup K(n)$ is a subcontinuum
%$and is manifestly decomposable since $r \ne r'$ implies neither of $K(r)$ or $K(r')$ contains the other.

Define the subcontinuum $Y = K(1) \cup K(2) \cup \ldots \cup K(n)$
%Since $p$ properly disrupts $\cU$ all $U_1, \ldots U_n$ are pairwise disjoint.
%It follows $K(r)$ does not contain $K(m)$ for $r \ne m$.
and write $C(x)$ for the continuum component of each $x \in X-Y$.
We claim each $C(x)$ meets no $U_r$. 
For suppose otherwise.
Then without loss of generality some $C(x)$ meets $U_1$.
Since $p$ disrupts $\{U_1, \ldots U_n\}$ we know $C(x)$ fails to meet at least one $U_m$.
Without loss of generality $m=2$.
Then $\overline{C(x)}$ is a subcontinuum with $x \in \overline{C(x)} \subset X-U_2$.

Now consider the subcontinuum $K(2) \cup \overline{C(x)}$.
We have $p \in K(2) \cup \overline{C(x)} \subset X-U_2$ and $U_m \subset K(2) \cup \overline{C(x)}$ for each $m \ne 2$.
Since we chose $K(2)$ to be maximal with both properties we must have $K(2) \cup \overline{C(x)} = K(2)$.
Hence $\overline{C(x)} \subset K(2)$ and $x \in K(2) \subset Y$ contrary to the assumption that $x \in X-Y$.
We conclude $C(x)$ meets no $U_r$. 

It follows for each $x \in X-Y$ that $\overline {C(x)}$ is a subcontinuum with $\{p,x\} \subset \overline {C(x)} \subset X - (U_1 \cup \ldots \cup U_n)$.
Define the subcontinuum $C = $ $\overline{ \bigcup \{C(x): x \in X-Y\}}$.
Then $\{p,x\} \subset C \subset X - (U_1 \cup \ldots \cup U_n)$.

It follows $X = C \cup K(1) \cup K(2) \cup \ldots \cup K(n)$
and $p$ is an element of each element of the covering.
By discarding any covering element contained in the union of the others we get the required decomposition.
\end{proof}

The next theorem follows from Theorem \ref{theorem2} and Lemma \ref{ndecomp}.

\begin{thm3}\label{indsp}
Every point of an indecomposable continuum is a proper shore point.
\end{thm3}

The proof of Theorem \ref{indsp} for the special case of $\HH^*$ uses standard techniques from the study of the Stone-\v Cech remainder.
An earlier paper of ours gave a lengthy argument for $\HH^*$, that considered separately two types of composants of $\HH^*$ and did not use Theorem \ref{theorem2}.
That paper was never published because the anonymous referee was able to give the much simpler proof we include here as Lemma \ref{Anon}. The second half of the proof is based on \cite{NCF2} Theorem 4.1.

We briefly recall the terminology from \cite{Me2}:
The {\it Stone-\v Cech remainder} $\HH^*$ is the set of nonprincipal ultrafilters of closed sets on the half open interval $\HH = [0, \infty)$. The topology on $\HH^*$ is generated by the sets $$U^* = \{\cD \in \HH^*: D \subset U\text{ for some }D \in \cD\}$$ as $U$ ranges over all open subsets of $\HH$. Note for $U$ bounded we have $U^* = \0$.
It is known that $\HH^*$ is an indecomposable Hausdorff continuum.

Suppose we have a nonprincipal ultrafilter $\cD$ on $\W$ and sequence of intervals $I_n = [a_n,b_n]$ with each $b_n < a_{n+1}$. For each subset $D \subset \W$ we write $I_D = \bigcup\{I_n:n \in D\}$ for the subset of $\HH$. We write $\overline {I_D}$ for the closure in $\beta \HH$. That means the collection of ultrafilters $\cD$ on $\HH$ with $I_D \in \cD$. We write $\II_\cD$ for the subset $ \bigcap \big \{\HH^* \cap \overline{I_D} : D \in \cD \big \}$ of $\HH^*$. Sets of the form $\II_\cD$ are known to be proper subcontinua of $\HH^*$ and are called {\it standard subcontinua}. For background on $\HH^*$ see \cite{CS1}. For background on Stone-\v Cech compactifications in general see \cite{uff} and \cite{CSbook}. 

\begin{lemma}\label{Anon}
Let $p \in \HH^*$ be arbitrary and $\cU = \{U_1, \ldots, U_n\}$ a family of open sets.
For each standard subcontinuum $\II_\cD$ with $p \in \II_\cD$ some other standard subcontinuum $\JJ_\cD \subset \K(p)$
has $p \notin \JJ_\cD$ but $\JJ_\cD$ meets each $U_m$.
In particular each point of $\HH^*$ is a shore point.
\end{lemma}

\begin{proof}
We only consider $n=2$ and $\cU=\{U_1,U_2\}$ as the general case is similar.
First choose disjoint open $U,V \subset \HH$ with the basic open sets $U^* \subset U_1$ and $V^* \subset U_2$.
Then observe $U= A_1 \cup A_2 \cup \ldots$ is a disjoint union of open intervals and likewise for $V= B_1 \cup B_2 \cup \ldots$ .

Choose an increasing sequence $c_1<c_2 < \ldots$ in $\HH$ so each $[c_n,c_{n+1}]$ contains at least one $A_i$ and $B_j$.
It follows $c_n \to \infty$.
Define sequences $R_n = [c_{4n}, c_{4n+1}]$ and $L_n = [c_{4n+2}, c_{4n+3}]$.
By construction $\RR_\cD$ and $\LL_\cD$ are disjoint and each meets $U^*$ and $V^*$.
Therefore $p$ is an element of at most one of $\RR_\cD$ and $\LL_\cD$.
Without loss of generality $p \notin \RR_\cD$.

Taking $\JJ_\cD = \RR_\cD$ it remains to show $\JJ_\cD \subset \K(p)$.
In case $\JJ_\cD$ meets $\II_\cD$ the fact that $\HH^*$ is indecomposable says $\JJ_\cD \cup \II_\cD \ne \HH^*$ hence $\JJ_\cD \subset \K(p)$ as required. 
Otherwise $\II_\cD$ and $\JJ_\cD$ are disjoint. Recall 
$\II_\cD =\bigcap \big \{\HH^* \cap \overline{I_D} : D \in \cD \big \}$ and 
$\JJ_\cD =\bigcap \big \{\HH^* \cap \overline{J_D} : D \in \cD \big \}$ are intersections of closed subsets of the compact $\HH^*$. Since they are disjoint the Cantor property says $\HH^* \cap \overline{I_A}$ and $\HH^* \cap \overline{J_B}$ are disjoint for some $A,B \in \cD$. 

It follows from the definition of $\overline{I_A}$ and $\overline{I_B}$ that $I_A \cap J_B \subset \HH$ is compact. Since $A \cap B \in \cD$ and $I_{A \cap B} \cap J_{A \cap B} \subset I_A \cap J_B$ we also see $I_{A \cap B} \cap J_{A \cap B}$ is compact. That means the set

$$ C = \{n \in A \cap B: I_n \text{ meets } J_m \text{ for some } m \in A \cap B\}$$

is finite. Since $\cD$ is nonprincipal $C \notin \cD$ and so $D = (A \cap B - C) \in \cD$.

According to the definiton of $\II_\cD$ as in intersection, the subcontinuum $\II_\cD$ does not change if we redefine the intervals $I_n$ for all $n \in D^c$. Likewise for $\JJ_\cD$.
Hence we can redefine $I_n,J_n$ for $n \in D^c$ without changing $\II_\cD$ and $\JJ_\cD$ and hence assume $\II_\W \cap \JJ_\W$ is empty. 
Then we can combine the two sequences into some $\{K_n: n \in \W \}$ where each interval $K_n$ is some $I_m$ or $J_m$ and each $K_{n+1}$ is to the right of $K_n$. 

Define the maps $\A: \W \to \W$ by letting each $\A(n)$ be the unique $i \in \W$ with $I_n = K_i$. In other words $I_n = K_{\A(n)}$. Likewise define $\B: \W \to \W$ by $J_n = K_{\B(n)}$. It follows from the definitions that $\II_\cD = \mathbb K _{\A(\cD)}$ and $\JJ_\cD = \mathbb K _{\B(\cD)}$ where we define the ultrafilter $\A(\cD) = \{E \subset \W: \A^{-1}(E) \in \cD\}$ and likewise for $\B(\cD)$. 

Define a finite-to-one function $g$ separately over the disjoint sets $A= \{\A(n): n \in \W\}$ and $B= \{\B(n): n \in \W\}$ by $g(\A(n))= \A(n)$ and $g(\B(n))= \A(n)$. The definition $g(\B(n))= \A(n)$ makes sense because $\B$ is injective. From here it quickly follows $g(\A(\cD)) = \A(\cD) = g(\B(\cD))$. 	
From \cite{NCF1} Lemma 10 and the paragraph before that lemma, if there exists a finite-to-one function \mbox{$g: \W \to \W$} with $g(\A(\cD)) = g(\B(\cD))$ then there also exists a finite-to-one monotone function $f: \W \to \W$ with $f(\A(\cD)) = f(\B(\cD))$. 

Since $f$ is finite-to-one each $f^{-1}(n)$ is finite. Thus the convex hull $L_n$ of $\{K_m: m \in f^{-1}(n)\}$ is a closed interval. Since $f$ is monotone each $L_{n+1}$ is to the right of $L_n$. So we can define a standard subcontinuum $\LL_{f(\A(\cD))}$. We claim  $\KK_{\A(\cD)},\KK_{\B(\cD)} \subset \LL_{f(\A(\cD))}$. Since $\JJ_\cD = \KK_{\B(\cD)}$ this shows $\JJ_\cD  \subset \K(p)$.

To that end write $\A(\cD) = \cU$. First observe each $K_n \subset L_{f(n)}$ since $L_{f(n)}$ is the hull of $\{K_m: m \in f^{-1}(f(n))\}$ and we have $m \in f^{-1}(f(n))$ for $m=n$.  It follows $K_U \subset L_{f(U)}$ for each $U \in \cU$. Hence we can write

$$\KK_\cU = \bigcap_{U \in \cU} \HH^* \cap \overline{K_U} \subset \bigcap_{U \in \cU} \HH^* \cap \overline{L_{f(U)}}.$$

It follows from the definition of $f(\cU)$ and $\cU$ being an ultrafilter that $f(U) \in f(\cU)$ for each $U \in \cU$. In the other direction, for each $E \in f(\cU)$ we have $f^{-1}(E) \in \cU$ and so $E$ contains the set $f(U)$ for $U = f^{-1}(E)$. It follows $\{f(U): U \in \cU\}$ is a cofinal subset of $f(\cU)$. Hence the intersection on the right-hand-side equals

$$\bigcap_{W \in f(\cU)} \HH^* \cap \overline{L_{W}} = \LL_{f(\cU)}.$$

We conclude $\KK_{\A(\cD)} \subset \LL_{f(\A(\cD))}$. By symmetry the same holds for $\KK_{\B(\cD)}$. This completes the proof.
\end{proof}

We would like to show our example of \textit{spot-welding} two copies of $\HH^*$ is generic in the following sense.

\begin{cnj3}
Suppose $X$ has no shore points and $p \in X$.
There is a decomposition $X = X_1  \oplus \ldots \oplus X_N$ with $p \in X_1 \cap \ldots \cap X_N$
and $p$ non-coastal when treated as a point of each $X_n$.
\end{cnj3}

Conjecture 4 asks for a particularly \textit{nice} decomposition of $X$. One variant of the conjecture is that every decomposition with $p \in X_1 \cap \ldots \cap X_N$ is \textit{nice}. The stronger conjecture however is false.

For a counterexample assume NCF and take two copies $H_1$ and $H_2$ of $\HH^*$ and 
identify points $x_1 \in H_1$ and $x_2 \in H_2$ respectively
with the endpoints $0$ and $1$ of the arc.
Denote by $X$ the quotient space.

Observe for $X_1 = H_1 \cup [0,1]$ and $X_2 = [0,1] \cup H_2$ we have the decomposition $X = X_1 \oplus X_2$ but the semicontinua $H_1 \cup [0, 1)$ and $ (0, 1] \cup H_2$
witness how $p = 1/2$ is a coastal point of each element.
Thus the stronger conjecture fails.

To see the weaker conjecture holds consider the second choice of decomposition 
$X_1 = H_1 \cup [0,1/2]$ and $X_2 = [1/2,1] \cup H_2$. By the same reasoning as in Theorem 1 we see $p=1/2$ is a coastal point of neither element.

The second decomposition above is \textit{minimal} in the following sense.

\begin{definition}
Suppose $X = X_1  \oplus \ldots \oplus X_N$ and $X = Y_1  \oplus \ldots \oplus Y_N$ are decompositions. We write $Y_1  \oplus \ldots \oplus Y_N \le X_1  \oplus \ldots \oplus X_N$ to mean there is a permutation $\sigma$ 
of $\{1,2,\ldots, N\}$ with each  $Y_n \subset X_{\sigma(n)}$.
We call a decomposition \textit{minimal} to mean it is minimal with respect to this partial order.
\end{definition}

For example, each minimal decomposition $X_1 \oplus X_2$ of the arc has $X_1 \cap X_2$ a singleton.
Each minimal decomposition $X_1 \oplus X_2$ of the circle has $X_1 \cap X_2$ a doubleton.
For $X$ formed by \textit{spot-welding} finitely many indecomposable continua the natural decomposition is minimal.
It may prove useful that minimal decompositions always exist.

\begin{lemma}
Each decomposition is $\le$-above some minimal decomposition.
\end{lemma}

\begin{proof}
The proof uses Zorn's lemma.
Suppose $\{X^i(1)  \oplus \ldots \oplus X^i(N): i \in \cI\}$ is a chain of decompositions.
Without loss of generality $\cI$ has top element $1 \in \cI$ and no bottom element.
Since $X^1(1)  \oplus \ldots \oplus X^1(N)$ is a decomposition there are points $x_n \in X^1(n) - \bigcup \{X^1(m):m \ne n\}$.

Let $X^i(1)  \oplus \ldots \oplus X^i(N)$ be arbitrary.
Since $i \le 1$ some permutation $\si_i$ has each $X^i \big (\si_i(n) \big ) \subset X^1 (n)$.
Thus $X^i \big (\si_i(n) \big )$ includes at most the element $x_{n}$ of $ \{x_1,x_2,\ldots,x_N\}$.
But since $X^i(1)  \cup \ldots \cup X^i(N) = X $ each $x_n$ is an element of some $X^i(m)$.
We conclude each $x_{n} \in X^i \big (\si_i(n) \big ) - \bigcup \big \{  X^i \big (\si_i(m) \big ) : m \ne n\big \}$.

For $j \le i$ we know $X^j \big (\si_j(n) \big )$ is contained in some $X^i \big (\si_i(m) \big )$.
The point $x_n$ witnesses how $m=n$.
We conclude each $X^j \big (\si_j(m) \big ) \subset X^i \big (\si_i(m) \big )$.
Then \cite{nadlerbook} Proposition 1.7 says the intersection $X_n =$ $\bigcap \big \{ X^i\big (\si_i(n)\big ):i \in I \big \}$ is a subcontinuum.
We claim $X_1, \ldots ,X_N$ form a decomposition which is clearly a lower bound for the chain.
First observe each $x_n$ witnesses how $X_n$ is not contained in $\bigcup \{X(m):m \ne n\}$.

To show $X_1 \cup \ldots \cup X_N = X$ let $x \in X$ be arbitrary.
For each $i \in \cI$ there is $n_i \le N$ with $x \in X^i\big (\si_i(n_i)\big )$.
Write $A(n) = \big \{i \in \cI: x \in X^i\big (\si_i(n)\big )\big \}$ for $n=1,2,\ldots, N$.
Since $\cI$ has no bottom element one of $A(n)$ is cofinal.
Thus $x \in \bigcap \big \{X^i\big (\si_i(n)\big ):i \in A(n)\big \}$ which equals $ 
\bigcap \big \{X^i\big (\si_i(n_i)\big ):i \in \cI\big \} $ by cofinality and thus $x \in X_n$.
Since $x$ is arbitrary we see $X_n$ cover $X$.

We conclude the chain $\{X^i(1)  \oplus \ldots \oplus X^i(N): i \in \cI\}$ has a lower bound $X_1 \oplus\ldots \oplus X_N$.
Zorn's lemma then implies each decomposition is above a minimal decomposition.
\end{proof}

\begin{cnj3}
Suppose $X$ has no shore points and $X = X_1  \oplus \ldots \oplus X_N$ 
is a minimal decomposition with $p \in X_1 \cap \ldots \cap X_N$.
Then $p$ is non-coastal when treated as a point of each $X_n$.
\end{cnj3}

Thus far we have only the partial results Lemmas \ref{5L1}, \ref{5L2} and \ref{5L3}.
Henceforth assume $X$ has no shore points. Fix $p \in X$ and let $X=X_1 \oplus X_2$ be a decomposition with $p \in X_1 \cap X_2$.
The case for $n > 2$ is similar.

\begin{lemma}\label{5L1}
Each dense semicontinuum of $X_1$ (resp. $X_2$) at $p$ contains \mbox{$X_1-X_2$ (resp. $X_2 - X_1$)}.
\end{lemma}

\begin{proof}
Suppose for example $p \in S \subset X_1 - q$ for some dense semicontinuum $S \subset X_1$ and $q \in X_1 - X_2$.
Then $X_2 \cup S \subset X-q$ is a dense semicontinuum of $X$.
This implies $q \in X$ is a non-block point hence a shore point contrary to assumption.
\end{proof}

\begin{corollary}\label{singleton}
Suppose $X_1 \cap X_2 = \{p\}$. Then $p$ is non-coastal as an element of $X_1$ and $X_2$.
\end{corollary}

The following notation is part of Lemma \ref{5L2}

\begin{notation}
Suppose $X=X_1 \oplus X_2$. 
Write $\cS(X_1)$ (resp. $\cS(X_2)$) for the collections of proper dense semicontinua of $X_1$ (resp. $X_2$)
that meet both \mbox{$X_1 \cap X_2$} and $X_1 - X_2$ (resp. $X_2 - X_1$).
Define two subsets of $X_1 \cap X_2$.

\begin{center}
	$C_1 =  \{x \in X_1 \cap X_2: x \in S$ for each $S \in \cS(X_1)\}$.
	\vspace{2mm}
	
	$C_2 = \{x \in X_1 \cap X_2: x \in S$ for each $S \in \cS(X_2)\}$.
\end{center}\vspace{2mm}
\end{notation}
\vspace{-7mm}

\begin{lemma}\label{5L2}
Suppose $p$ is coastal as an element of both $X_1$ and $X_2$. 
Then one of $C_1 \cap C_2= \0$ or $C_1 \cup C_2 = X_1 \cap X_2 $ holds.
\end{lemma}

\begin{proof}
Since $p \in X_1$ is coastal $\cS(X_1)$ is nonempty and likewise for $\cS(X_2)$.
Suppose $C_1 \cap C_2 \ne \0$ and $C_1 \cup C_2 \ne X_1 \cap X_2 $.
That means there are $x \in C_1 \cap C_2$ and $y \in X_1 \cap X_2 -C_1 \cup C_2$.
Select $S_1 \in \cS(X_1)$ and $S_2 \in \cS(X_2)$ with $y \notin S_1$ and $y \notin S_2$.

By definition we have $x \in S_1$ and $x \in S_2$.
Thus $S_1 \cup S_2 \subset X$ is a dense semicontinuum that excludes the point $y \in X$.
This contradicts how $X$ has no shore points.
We conclude $C_1 \cap C_2 \ne \0$ and so $C_1 \cup C_2 = X_1 \cap X_2$.
\end{proof}

In the first case of Lemma \ref{5L2} we can say more.

\begin{lemma}\label{5L3}
Suppose $p$ is coastal as an element of both $X_1$ and $X_2$ and \mbox{$C_1 \cap C_2= \0$}.
Then each subcontinuum of $X$ that meets $C_1$ and $C_2$ also contains $ X_1 \cap X_2 - C_1 \cup C_2 $.
\end{lemma}

\begin{proof}
Suppose the subcontinuum $K \subset X$ meets $C_1$ and $C_2$.
Let $x \in X_1 \cap X_2 - C_1 \cup C_2 $ be arbitrary.
Select $S_1 \in \cS(X_1)$ and $S_2 \in \cS(X_2)$ with $x \notin S_1$ and $x \notin S_2$.
It follows $S_1 \cup K \cup S_2 \subset X$ is a dense semicontinuum.
Since $X$ has no shore points $S_1 \cup K \cup S_2 = X$ and so $x \in K$.
Since $x \in  X_1 \cap X_2 - C_1 \cup C_2  $ is arbitrary we conclude $ X_1 \cap X_2 - C_1 \cup C_2  \subset K$.
\end{proof}

\section*{Acknowledgements}
This research was supported by the Irish Research Council Postgraduate Scholarship Scheme grant number GOIPG/2015/2744. 
The author would like to thank Professor Paul Bankston and Doctor Aisling McCluskey for their help in preparing the manuscript, and the anonymous referee for their attention and suggestions.

\end{document}